\documentclass[11pt, reqno]{amsart}

\usepackage{amssymb,latexsym,amsmath,amsfonts,amsthm}
\usepackage{mathrsfs}
\usepackage{enumitem}
\usepackage[usenames]{color}
\usepackage{hyperref}
\usepackage{comment}
\usepackage{mathtools}
\allowdisplaybreaks

\numberwithin{equation}{section}

\theoremstyle{definition}
\newtheorem{definition}{Definition}[section]

\theoremstyle{remark}
\newtheorem{remark}[definition]{Remark}

\theoremstyle{plain}
\newtheorem{theorem}[definition]{Theorem}
\newtheorem{result}[definition]{Result}
\newtheorem{lemma}[definition]{Lemma}

\newtheorem{example}[definition]{Example}

\newcommand{\D}{\mathbb{D}}
\newcommand\dball[3]{B_{{{#1}}^{{#2}}_{{#3}}}}

\newcommand{\hol}{\mathcal{O}}

\newcommand{\plog}{{\rm Log}}

\newcommand{\bcdot}{\boldsymbol{\cdot}}

\newcommand{\btl}{\blacktriangleleft}

\newcommand{\Z}{\mathbb{Z}}
\newcommand{\N}{\mathbb{N}}

\newcommand{\C}{\mathbb{C}} 
\newcommand{\R}{\mathbb{R}}

\definecolor{DPurple}{rgb}{0.46,0.2,0.69}

\begin{document}

\title[The Carath{\'e}odory topology]{Remarks on a result of Sibony on the\\ Carath{\'e}odory topology}

\author{Sudip Dolai}
\address{Department of Mathematics, Indian Institute of Science, Bangalore 560012, India}
\email{sudipdolai@iisc.ac.in}

\begin{abstract}
In this paper, we prove that if a Carath{\'e}odory hyperbolic analytic space $X$ is $C_X$-complete, then its
natural topology is induced by the Carath{\'e}odory distance on $X$. This is an 
improvement of Sibony's result, which concludes the same under the hypothesis that $X$ is $C_X$-finitely 
compact. This improvement is not merely formal; we also show the existence of uncountably many biholomorphically
inequivalent analytic spaces that are not $C_X$-finitely compact but are $C_X$-complete.
\end{abstract}

\keywords{Analytic spaces, Carath{\'e}odory distance, Carath{\'e}odory topology}
\subjclass[2020]{Primary: 32C18, 32F45, 51F99; Secondary: 30C20, 30L15}

\maketitle

\section{Introduction and statement of results}\label{S: intro}
This paper is motivated by the problem of finding a sufficiently general condition ensuring that, given
an analytic space $X$ on which the Carath{\'e}odory pseudodistance is a true distance, the natural topology on
$X$ is induced by the latter distance. This is a very useful condition (it is, in part, the reason why the
Kobayashi distance is so useful) but it is not true in general even for domains in $\C^d$ when $d$ is
sufficiently large\,---\,see \cite[Section~2.4]{jarnickipflug:idca93}.
\smallskip

Before proceeding further, we define some of the terms in the previous paragraph. We will not, however,
explicitly define analytic spaces because they constitute such a general class of objects that their
definition is somewhat involved\,---\,instead, we refer the reader to
\cite[Chapter~V, Section~A]{gunningRossi:afscv65}. Loosely speaking, an analytic space $X$ is a Hausdorff
topological space such that for each point $a\in X$, there exists an open neighbourhood $U_a$ of $a$ and
a number $k(a)\in \Z_+$ that is isomorphic, in an appropriate sense, to a complex-analytic subvariety
of a domain $\Omega_a\subseteq \C^{k(a)}$. We shall call such a neighbourhood a \emph{model neighbourhood}.
We shall refer to the topology on $X$ as the \textit{natural topology} or the \textit{analytic topology}
on $X$, which we will denote by ${\rm top}(X)$. Next, we give the following (in what follows, $\D$ will denote the open unit disc with centre $0\in \C$ and $\mathcal{O}(X, \D)$ will denote the set of all holomorphic functions from $X$ to $\D$):

\begin{definition}\label{D: cara}
Let $X$ be an analytic space. The \textit{M{\"o}bius pseudodistance} for $X$ is defined as 
\begin{equation*}
  C_X^*(x,y):=\sup\left\{\,\left|\frac{f(y)-f(x)}{1-\overline{f(x)}f(y)}\right|: 
  f\in \mathcal{O}(X, \D)\right\}
  =\sup\{|f(y)|: 
  f\in \mathcal{O}(X, \D), \ f(x)=0\},
  \end{equation*}
for any $x,y\in X$.
The \textit{Carath{\'e}odory pseudodistance} for $X$ is defined as 
\[
  C_X(x,y):= \tanh^{-1}(C_X^*(x,y))
\]
for any $x,y\in X$. 
\end{definition}

\begin{remark}
The relationship between $C_X^*$ and $C_X$ arises from the fact that, in defining the latter, the pseudohyperbolic distance
between $f(x)$ and $f(y)$ in the definition of $C_X^*(x,y)$ is replaced by the Poincar{\'e} distance between the latter.
\end{remark}

\begin{remark}\label{Rm: rem}
Since $\tanh^{-1}:[0,1) \to [0, \infty)$ is a monotone increasing continuous function, 
the pseudodistances $C_X$ and $C_X^*$ share many of the same fundamental properties. For example,
$C_X$ is a distance on $X$ if and only if $C_X^*$ is a distance; if either of $C_X$ or $C_X^*$
is a distance on $X$ then the topologies induced by them are the same; etc.
\end{remark}

For an analytic space $X$, we say that $X$ is \textit{Carath{\'e}odory hyperbolic} if $C_X$ is a distance on $X$.
In order to state the results below, we need another set of definitions.

\begin{definition}\label{D: complete}
Let $X$ be a Carath{\'e}odory hyperbolic analytic space. Then:
\begin{itemize}
  \item[$(1)$] $X$ is said to be \textit{$C_X$-complete} if every $C_X$-Cauchy sequence $(x_n)_{n\geq 
  1}\subset X$ converges with respect to ${\rm top}(X)$ to a point $x\in X$, and 
  \item[$(2)$] $X$ is said to be \textit{$C_X$-finitely compact} if for each $x\in X$ and $r>0$ the 
  Carath{\'e}odory ball $\dball{C}{}{X}(x,r)$ (see Section~\ref{SS: nota_conv}) is relatively compact with 
  respect to ${\rm top}(X)$ inside $X$.  
\end{itemize}
\end{definition}

We must clarify that although the term ``$C_X$-complete" may suggest that it refers to $(X, C_X)$ being
Cauchy-complete, it indicates a formally stronger property than the latter. The terminology in 
Definition~\ref{D: complete} coincides with that in \cite{jarnickipflugvigue:aeccbnfca93} by 
Jarnicki--Pflug--Vigu{\'e}, which inspires some of the arguments in this paper.  
\smallskip

A general sufficient condition for the property referred to at the top of this section is given by 
Result~\ref{R: sibony} by Sibony. In the following result, we denote by ${\rm top}(C_X)$ the topology induced by the distance $C_X$.

\begin{result} [paraphrasing {\cite[Corollaire~4]{sibony: pdfhbmc75}}]\label{R: sibony}
Let $X$ be a Carath{\'e}odory hyperbolic analytic space that is $C_X$-finitely compact. Then, ${\rm top}(X)={\rm top}(C_X)$.
\end{result}

\begin{remark}
In Sibony's notation in \cite{sibony: pdfhbmc75}, being $C_X$-finitely compact is expressed as ``$C_X^{H^{\infty}(X)}$ complete".
That the two terminologies are the same is evident from \cite[D{\'e}finition~3]{sibony: pdfhbmc75}.
\end{remark}

It turns out that if an analytic space $X$ is $C_X$-finitely compact, then it is $C_X$-complete. Moreover, this
inclusion is strict, as shown by \cite{jarnickipflugvigue:aeccbnfca93}. Our first result improves upon
Result~\ref{R: sibony}; it yields the same conclusion as the latter under a weaker condition. We will elaborate on
this presently, but let us now state our first result.

\begin{theorem}\label{T: strngly_complte}
Let $X$ be an analytic space that is $C_X$-complete. Then, ${\rm top}(X)= {\rm top}(C_X)$.    
\end{theorem}

The example in \cite{jarnickipflugvigue:aeccbnfca93} shows that Theorem~\ref{T: strngly_complte} is a genuine
improvement of Result~\ref{R: sibony}; it provides an example of an analytic space $X$ to which
Result~\ref{R: sibony} is inapplicable because it is not $C_X$-finitely compact,
but for which Theorem~\ref{T: strngly_complte} shows that ${\rm top}(X)= {\rm top}(C_X)$. However,
the question now arises: \textit{how abundant is the class of analytic spaces $X$ that are $C_X$-complete
but \textbf{not}
$C_X$-finitely compact?}. Our next result implicitly answers this question.

\begin{theorem}\label{T: uncountable_example}
There exist uncountably many biholomorphically inequivalent Carath{\'e}odory hyperbolic analytic spaces that are
not $C_X$-finitely compact but on each of which the metric topology induced by the Carath{\'e}odory distance 
coincides with its natural topology.    
\end{theorem}

The proof of Theorem~\ref{T: uncountable_example} relies on a class of examples that are presented in
Section~\ref{S: example}. The proofs of the above theorems are given in Section~\ref{S: pf_thm}.
\smallskip

\section{Topological preliminaries}\label{S: top_pre}
In this section, we present a lemma that forms the core of the proof of Theorem~\ref{T: strngly_complte}.
But first, we present some notation and conventions, which will be needed below and in later sections.  

\subsection{Notation and conventions}\label{SS: nota_conv}
In the following list, $X$ is a Carath{\'e}odory hyperbolic analytic space.
\begin{enumerate}
  \item[$(1)$] In view of Remark~\ref{Rm: rem}, $C_X^*$ too is a distance on $X$. The topology induced by it is denoted by ${\rm top}(C_X^*)$.  
  \item[$(2)$] We will write 
\[
  x_n \overset{X}{\longrightarrow} x \quad \text{and} \quad x_n \overset{C_X}{\longrightarrow} x
\]
to indicate the convergence of a sequence $(x_n)_{n\geq 1} \subset X$ with respect to ${\rm top}(X)$ and ${\rm top}(C_X)$, respectively. 
  \item[$(3)$] In our proofs below, it is more convenient to work with the M{\"o}bius distance in some 
  instances (in view of Remark~\ref{Rm: rem}, this does not affect our conclusions). Let $d_X$ denote either 
  $C_X^*$ or $C_X$. Then, for $x\in X$ and $r>0$, $\dball{d}{}{X}(x, r)$ is defined as 
  \[
    \dball{d}{}{X}(x, r):=\{y\in X: d_X(x, y)<r\}.
  \]
  \item[$(4)$] Given a point $x \in \C^d$ and $r>0$, $B^d(x,r)$ denotes the open Euclidean ball in
  $\C^d$ with radius $r$ and centre $x$. When $d=1$, for simplicity, we will write $B(x, r):=B^1(x,r)$.
  \item[$(5)$] In what follows, if $S \subseteq X$, then $S^{\mathrm{o}}$ will denote the interior of $S$ with
  respect to ${\rm top}(X)$. Similarly, if a topological condition is ascribed to $S$ without any further 
  qualification, this condition will be understood to hold with respect to ${\rm top}(X)$.
\end{enumerate}

\subsection{A key lemma} We now provide the statement and proof of a lemma central to the proof of Theorem~\ref{T: strngly_complte}. But first, we must explain some of terminology used in
the proof of this lemma. To do so, we will refer to the definitions in
\cite[Chapter~V, Section~A]{gunningRossi:afscv65} and to notation in Section~\ref{S: intro}. Given an
analytic space $X$ and a point $a\in X$, let $U_a$ be a model neighbourhood of $a$ as introduced in
Section~\ref{S: intro}. This means that there exist a $k(a)\in \Z_+$, a domain $\Omega_a\subseteq 
\C^{k(a)}$, an analytic subvariety $Y^a$ of $\Omega_a$, and a homeomorphism $f_a: U_a\to Y^a$
such that, for each $x\in U_a$, a $\C$-valued function $h$ defined in neighbourhood of $x$ is
declared to be holomorphic \emph{precisely} if $h = H\circ (f_a|_{{\sf dom}(h)})$, where $H$ is a
representative of some germ at $y:=f_a(x)$, say $\boldsymbol{H}_y\in {}_{Y^a}\hol_y$ and 
${}_{Y^a}\hol$ is the sheaf of germs of holomorphic functions on $Y^a$. In short, $X$ is modeled
locally as an analytic variety by $(Y^a, {}_{Y^a}\hol)$, which will be referred to as a
\emph{local model} in the proof below.
\begin{lemma}\label{L:equvlnce_condn}
Let $X$ be a Carath{\'e}odory hyperbolic analytic space. Then, the 
following conditions are equivalent:
\begin{enumerate}
  \item[$(1)$] For any compact subset $K \varsubsetneq X$ such that $K^{\mathrm{o}} \neq \emptyset$,
  given any $z \in K^{\mathrm{o}}$, there exists a number $r_z >0$ such that $\dball{C}{}{X}(z, r_z) \subseteq K$.
  \item[$(2)$] ${\rm top}(X) = {\rm top}(C_X)$.
\end{enumerate}
\end{lemma}
\begin{proof}
 Let $K \varsubsetneq X$ be a compact set with $K^{\mathrm{o}} \neq \emptyset $ and assume 
that for any $z \in K^{\mathrm{o}}$,  there exists $r_z >0$ such that $\dball{C}{}{X}
(z, r_z) \subseteq K$. Since, by definition, $C_X(a, \bcdot): X \to \R$ is a continuous function for 
all $a \in X$ (relative to the natural topologies on the domain and the range), ${\rm top}
(C_X) \subseteq {\rm top}(X)$. So we need to show that ${\rm top}(X) \subseteq {\rm top}
(C_X)$. Let $W$ be an open set in $X$; fix a point $a \in W$. Let $(U_a, f_a)$ be a model 
neighbourhood of $a$ with $U_a \subseteq W$ and let $(Y^a, {}_{Y^a}\hol)$ 
be the local model, where 
$k(a) \in \Z_+$ and $Y^a$ is a subvariety of a domain $\Omega_a \subseteq \C^{k(a)}$ with 
$\alpha = f_a(a)$. As $Y^a$ is locally compact, there exists $\varepsilon >0$ such that the
set $K_a:= \overline{B^{k(a)}(\alpha, \varepsilon)} \cap Y^a \varsubsetneq Y^a$ (see~\cite[Chapter~V, Section~A]{gunningRossi:afscv65} for an explanation of the notation). Note that $K_a$ is 
compact. As $f_a$ is a homeomorphism from $U_a$ to $Y^a$, we have $f_a^{-1}(K_a)$ is a compact set 
in $U_a$ with $a$ an interior point. By hypothesis, there exists $r_a>0$ such that $\dball{C}{}
{X}(a, r_a) \subseteq f_a^{-1}(K_a) \varsubsetneq U_a \subseteq W$. As $W$ and $a\in W$ were 
arbitrary, we have ${\rm top}(X) \subseteq {\rm top}(C_X)$.%
\smallskip

Now assume ${\rm top}(X) = {\rm top}(C_X)$. Let $K \subseteq X$ be a compact set such that
$K^{\mathrm{o}} \neq \emptyset$ and let $z \in K^{\mathrm{o}}$.

\medskip
\noindent{\textbf{Claim.} \emph{${\rm inf }\{C_X(x, z): x \in X \setminus K\} >0$}}.
\smallskip

\noindent{Suppose the claim is false. Then, there exists a sequence $(x_n)_{n \geq 1} \subset X 
\setminus K$} such that $x_n \overset{C_X}{\longrightarrow} z$.
As ${\rm top}(X) = {\rm top}
(C_X)$, it follows that $x_n \overset{X}{\longrightarrow} z$. As $x_n \in X 
\setminus K$ for all $n \in \Z_+$, the fact that
$x_n \overset{X}{\longrightarrow} z$ contradicts the
fact that $z\in K^{\mathrm{o}}$. Hence the claim.
\hfill$\btl$
\smallskip

By the above claim, there exists an $r_z >0$ such that $\dball{C}{}{X}(z, r_z) \subseteq K$. This
completes the proof.
\end{proof}
\smallskip

\section{Analytical preliminaries}\label{S: Tech_Lemma}
This section is devoted to several observations and results that play a vital role in the construction of Example~\ref{Ex: annulus}. We begin with the following definition.
\begin{definition}
Let $U\varsubsetneq\C$ be a bounded domain none of whose boundary components is a single point. Fix a point $p\in U$ and consider the extremal problem 
\[
  \max\{|g'(p)|: g\in\mathcal{O}(U,\D)\}.
\]
The solution to this problem exists uniquely, call it $F$. The function $F$ is said to be the \textit{Ahlfors function} for $U$ and $p$.
\end{definition}

Now we present a result that is used in Step~$4$ in Section~\ref{S: example}.
\begin{result}[{\cite[Theorem~3.1.12]{krantz:tgtcv25}}]\label{R: Ahlfors}
Let $U\Subset\C$ be a domain bounded by $m+1$ disjoint analytic simple closed curves $\Gamma_0, \dots, \Gamma_m$, let $p\in U$, and let $F$ be the 
Ahlfors function for $U$ and $p$. Then, $F$ maps $U$ onto $\D$ exactly $m+1$ times and $F$ extends analytically 
over each $\Gamma_j$ and maps each $\Gamma_j$ homeomorphically onto $\partial\D$. 
\end{result}

We shall now present several observations and lemmas. We begin with the properties of a holomorphic covering map from $\D$ to an annulus.

\begin{lemma}\label{L: covering_map_lemma}
Let $A(R)$ denote the annulus $\{w\in \C: 1<|w|<R\}$.
Then,
\begin{enumerate}
  \item [$(1)$] $p:\D \to A(R)$ defined by
  \begin{equation}\label{E: covering_map}
    p(z):=\sqrt R\exp\bigg(-\frac{i}{\pi}\log R \bcdot\plog \Big(\frac{1-z}{1+z}\Big)\bigg),
  \end{equation}
where $\plog$ denotes the principal holomorphic branch of the logarithm, is a covering map that maps $0$ to $\sqrt R$.
\item[$(2)$] For each $m\in \N\setminus\{0,1\},$
  \[
    x(m):=\frac{1-\exp\left(i(\tfrac{\pi}{2}-\tfrac{\pi}{m})\right)}
    {1+\exp\left(i(\tfrac{\pi}{2}-\tfrac{\pi}{m})\right)} \in p^{-1}\{R^{1-\tfrac{1}{m}}\}.
  \]
  \item[$(3)$] $|x(m)|^2=\big(1-\sin(\pi/m)\big)\big/ \big(1+\sin(\pi/m)\big)$ for each $m\in
  \N\setminus\{0,1\}$.
\end{enumerate}
\end{lemma}
\begin{proof}
The desired covering map can be constructed by first mapping $\D$ conformally to the right half
plane, then mapping the latter conformally to the strip $\{x+iy\in \C: -\pi/2 <y<\pi/2\}$, which 
can be mapped by a homothety and a translation onto the strip
\begin{equation*}
  \mathcal{S}:=\{x+iy\in \C: 0<x<\log R\}.
\end{equation*}
The exponential map, clearly, serves as a covering map $\mathcal{S}\to A(R)$. The right-hand side
of equation~\eqref{E: covering_map} represents the sequence of compositions of the standard mappings 
hinted at in the above description. This establishes $(1)$.
\smallskip

We will show that $p(x(m))=R^{1-\tfrac{1}{m}}$ for each $m\in \N\setminus\{0,1\}$. Fix an
$m_0\in \N\setminus\{0,1\}$ and substitute $z=x(m_0)$ into equation~\eqref{E: covering_map}. As the map
$\C\setminus\{-1\}\ni z \longmapsto (1-z)/(1+z)$ is an involution and $\plog$ is the
principal holomorphic branch of the logarithm, we have
\begin{equation*}
p(x(m_0))= \sqrt R \exp\big((\tfrac{1}{2}-\tfrac{1}{m_0})\log R\big). \end{equation*}
Using the identity $\exp(a\log b)=b^a$ for each $a\in \R$ and $b>0$ gives $p(x(m_0))=R^{1-\tfrac{1}
{m_0}}$. As $m_0 \in \N\setminus\{0,1\}$ was arbitrary, this completes the proof of $(2)$.
\smallskip

Finally, $(3)$ follows from the expression defining $x(m)$ by an elementary computation.
\end{proof}

Now we derive an upper bound of $|x(m)|$ for sufficiently large $m\in \Z_+$. This bound plays a crucial 
role in Step~$4$ in Section~\ref{S: example}.

\begin{lemma}\label{L: upperbound}
With $x(m)$ as defined in the statement of Lemma~\ref{L: covering_map_lemma}, there exists $m_1 
\in \Z_+$ such that 
\begin{equation}\label{E: upper_bound}
  C^*_{\D}(0, x(m)) = |x(m)|\leq 1-\tfrac{2}{m+1} \quad\forall m \geq m_1.
\end{equation}
\end{lemma}
\begin{proof}
We first observe that
\begin{equation}\label{E: limit_for_ub}
  \lim_{m\to \infty}   (m+1)(1-|x(m)|^2)=2\pi.
\end{equation}
It is elementary to show that 
\begin{equation}\label{E: elementary_inequality}
  \frac{1-|x(m)|^2}{2}\leq (1-|x(m)|) \quad  \forall m\in \N\setminus\{0,1\}.
\end{equation}
As $4<2\pi$, from~\eqref{E: limit_for_ub}, there exists $m_1\in \Z_+$ such that $(m+1)(1-|x(m)|^2) 
\geq 4$ for all $m\geq m_1.$ Now combining this with~\eqref{E: elementary_inequality} we have 
\[
  (m+1)(1-|x(m)|)\geq 2 \quad\forall m\geq m_1,
\]
 which is precisely~\eqref{E: upper_bound}.
\end{proof}

We conclude this section with the following lemma. This lemma provides a lower bound of $C^*_{A(R)}\big(\sqrt R, R^{1-
\frac{1}{m}}\big)$ for sufficiently large $m\in \Z_+$. We will use this in Step~$1$ in Section~\ref{S: example}.

\begin{lemma}\label{L: lowerbound}
 With $A(R)$ as defined in Lemma~\ref{L: covering_map_lemma}, there exists a constant $K(R)>0$ that depends 
 only on $R>1$ such that there exists $m_2\in \Z_+$ such that 
 \[
  C^*_{A(R)}\big(\sqrt R, R^{1-\tfrac{1}{m}}\big)\geq \Big(1-\tfrac{K(R)}{m}\Big) \quad \forall m 
 \geq m_2.
 \]
\end{lemma}
\begin{proof}
Consider the auxiliary map $f: A(R) \to \D$ defined by $f(w)=w/R$. Then, as $f$ is holomorphic, hence 
contractive for Carath{\'e}odory distances,
\begin{equation}\label{E: cara_lower_bnd}
  C^*_{A(R)}\big(\sqrt R, R^{1-\tfrac{1}{m}}\big) \geq \left|\frac{\sqrt R-R^{\tfrac{1}{m}}}{\sqrt
  RR^{\tfrac{1}{m}}-1}\right| \quad \forall m\in \Z_+ \setminus\{1\}.     
\end{equation}
For $R>1$, let us fix $K(R):=2\tfrac{\sqrt R +1}{\sqrt R -1}\log R$. As $R>1$, whenever $m\geq 3$, the 
expression within $|\bcdot|$ on the right-hand side of~\eqref{E: cara_lower_bnd} is positive. We work with 
this expression (in what follows, it will be assumed that $m\geq 3$) to get 
\begin{equation}\label{E: lower_bnd_cal}
  \frac{\sqrt R-R^{\tfrac{1}{m}}}{\sqrt RR^{\tfrac{1}{m}}-1}-\Big(1-\tfrac{K(R)}{m}\Big) 
  =\frac{m\big(\sqrt R-R^{\tfrac{1}{m}}\big)-m\big(\sqrt RR^{\tfrac{1}{m}}-1\big)+K(R)\big(\sqrt R 
  R^{\tfrac{1}{m}}-1\big)}{m\big(\sqrt R R^{\tfrac{1}{m}}-1\big)}.
\end{equation}
Define the auxiliary function $\tau: [3, \infty) \to \R$ as follows:
\[
  \tau(t):=t(\sqrt R+1)(1-R^{\tfrac{1}{t}}),
  \]
whose connection with~\eqref{E: lower_bnd_cal} will soon be clear. An appeal to L'Hospital's Rule tells us that
\begin{equation}\label{E: L'hospital_limit}
  \lim_{t\to \infty} \tau(t)=-(\sqrt R+1)\log R.
\end{equation}

Since, for $R>1$, we have
\[
  -\frac{3}{2}(\sqrt R+1)\log R< -(\sqrt R+1)\log R,
\]
from~\eqref{E: L'hospital_limit} there exists $m_2\in \Z_+$ (let us take $m_2\geq3$) such that
\begin{equation}\label{E: limit_lower_bound}
  \tau(t)\geq -\frac{3}{2}(\sqrt R+1)\log R \quad \forall t\geq m_2.
\end{equation}
Observe that the numerator of the right-hand side of~\eqref{E: lower_bnd_cal} can be written as 
\begin{equation}\label{E: numertaor_limit}
  \tau(m)+K(R)(\sqrt R R^{\tfrac{1}{m}}-1) \quad \forall m\geq 3.
\end{equation}
Now substitute the value of $K(R)$ in~\eqref{E: numertaor_limit} and note that $(\sqrt RR^{\tfrac{1}{m}}-1)/(\sqrt 
R -1)\geq 1$
for each $m\geq 3$. Combining this with~\eqref{E: limit_lower_bound} we have 
\begin{equation*}
  \tau(m)+K(R)(\sqrt R R^{\tfrac{1}{m}}-1) 
  \geq -\frac{3}{2}(\sqrt R+1)\log R +2(\sqrt R+1)\log R >0 \quad \forall m \geq m_2.
\end{equation*}
This shows that the numerator of the right-hand side of~\eqref{E: lower_bnd_cal} is positive for each $m \geq
m_2$. As $m(\sqrt RR^{\tfrac{1}{m}}-1)$ is positive for each $m\in \Z_+$, we have 
\begin{equation}\label{E: lower_bound}
  \frac{\sqrt R-R^{\tfrac{1}{m}}}{\sqrt RR^{\tfrac{1}{m}}-1}\geq\Big(1-\tfrac{K(R)}{m}\Big) \quad 
  \forall m\geq m_2. 
\end{equation}
Combining~\eqref{E: lower_bound}~and~\eqref{E: cara_lower_bnd}, we have the desired inequality. 
\end{proof}
\smallskip

\section{An example}\label{S: example}
This section is dedicated to constructing an example of an analytic space that is not biholomorphically
equivalent to the example constructed by Jarnicki--Pflug--Vigu{\'e} that was alluded to in
Section~\ref{S: intro}. This example will be central to the proof of Theorem~\ref{T: uncountable_example}.

\begin{example}\label{Ex: annulus}
An example of an analytic space that is Carath{\'e}odory hyperbolic, $C_X$-complete but not
$C_X$-finitely compact, and which is not biholomorphically equivalent to the example in
\cite{jarnickipflugvigue:aeccbnfca93}.
\end{example}
The construction of the above example and the proof of the assertions made in the statement of 
Example~\ref{Ex: annulus} will require several steps. Steps 1 and 4 are inspired in a significant way by the construction in \cite{jarnickipflugvigue:aeccbnfca93}.

\medskip
\noindent{\textbf{Step 1.}} \emph{Construction of the analytic space $X(R)$}
\smallskip

\noindent{We will begin by describing the topology of $X(R)$. To this end, let us abbreviate $A(R)$\,---\,the annulus 
discussed in Section~\ref{S: Tech_Lemma}\,---\,as $\Omega$. Let us consider the following sets in $\Omega$:
\begin{equation}\label{E: attaching_points}
  \mathscr{S}_{n}:=\{R^{1-\tfrac{1}{m}}: m=2^n, 2^n+1,\dots,2^{n+1}-1\}
\end{equation}
for $n\in \Z_+$. For each $n\in \Z_+$, let us write $x^{(n)}_m=R^{1-1/(2^{n}+m-1)}$, $m=1,\dots,2^n$. As a 
topological space,
\[
  X(R):=\bigcup_{n=0}^\infty \Omega \times \{n\}\bigg/\!\!\thicksim
\]
where we define
\[
  (z,j)\thicksim (w,k) \iff
  \begin{cases} z=w \text{ and } j=k, \text{ or }\\
    \text{$j=0$, $k$}\in \Z_+ \text{ and } z=w=x^{(k)}_m \text{ for some } m=1 ,\dots, 2^k,
  \end{cases}
\]
and endow $X(R)$ with the quotient topology induced by the identification described above. More descriptively, the
set $X(R)$ and the topology on it are obtained by considering a disjoint union of countable copies of 
$\Omega$, labelling each copy as $\Omega_n, n\in\N$, viewing each set $\mathscr{S}_n$ as included in $\Omega_n, n\in \Z_+$, 
and identifying each $x^{(n)}_m$ with $y^{(n)}_m$, $m=1,\dots,2^n$, where
\[
  y^{(n)}_m:= \text{the point $R^{1-1/(2^{n}+m-1)}$ viewed as belonging to }\Omega_0.
\]
Note that $X(R)$ is Hausdorff, which is a fact that we will need in the argument that follows.
\smallskip

To specify the analytic structure on $X(R)$, we shall appeal to \cite[Proposition~7, page~150]{gunningRossi:afscv65}.
Let us first denote the above-mentioned quotient map by
$Q: \bigcup_{n=0}^\infty \Omega \times \{n\}\to X(R)$ and $Q((x,n))$ as $[x,n]$. Let $\alpha\in X(R)$.
Our analysis consists of two cases.

\medskip
\noindent{\textbf{Case~1.} \emph{$\alpha$ is such that $Q^{-1}\{\alpha\}$ is a singleton.}}
\smallskip

\noindent{Write $Q^{-1}\{\alpha\} = \{(x_0, n)\}$. Let $D_{\alpha}$ denote a disc with centre $x_0$ such that
$D_{\alpha}\varsubsetneq \Omega$ and
\begin{itemize}
  \item $D_{\alpha}\cap \mathscr{S}_n = \emptyset$ if $n\in \Z_+$,
  \item $D_{\alpha}\cap (\bigcup_{n\in \Z_+}\mathscr{S}_n) = \emptyset$ if $n=0$.
\end{itemize}
Let us define $X_{\alpha} := Q(D_{\alpha}\times \{n\})$ and
$\psi_{\alpha}: X_{\alpha}\ni [z,n]\longmapsto (z-x_0)$.}
  
\medskip
\noindent{\textbf{Case~2.} \emph{$\alpha$ is such that $Q^{-1}\{\alpha\}$ has two elements.}}
\smallskip

\noindent{By definition, there exist $n\in \Z_+$ and $m=1,\dots, 2^n$,
such that $Q^{-1}\{\alpha\} = \{(x^{(n)}_m, n), (x^{(n)}_m, 0)\}$. Let $D_{\alpha}$ denote a disc with
centre $x^{(n)}_m$ such that $D_{\alpha}\varsubsetneq \Omega$ and
$(D_{\alpha}\setminus \{x^{(n)}_m\})\cap (\bigcup_{n\in \Z_+}\mathscr{S}_n) = \emptyset$.
Let $X_{\alpha} = Q\big((D_{\alpha}\times \{n\})\cup (D_{\alpha}\times \{0\})\big)$. Now,
define $\psi_{\alpha}: X_{\alpha}\to \{(z_1, z_2)\in \C^2: z_1z_2=0\}$ as follows:
\[
  \psi_{\alpha}([z,j]) := \begin{cases}
                            (z-x^{(n)}_m, 0), &\text{if $j=0$}, \\
                            (0, z-x^{(n)}_m), &\text{if $j=n$}. 
                          \end{cases}
\]}

Note that, owing to the quotient topology, each $\psi_{\alpha}$ is a homeomorphism onto its image.
Clearly $X(R) = \bigcup_{\alpha\in X(R)}X_{\alpha}$. It is elementary to check that the compatibility
conditions stated in \cite[Proposition~7, page~150]{gunningRossi:afscv65} hold true. Then, this
proposition endows $X(R)$ with the structure of a connected one-dimensional analytic space.}
\smallskip

Before proceeding to the next step, we need to make some preparations. First: a holomorphic function $\phi: 
X(R)\to \C$ can be identified with the sequence of holomorphic functions $\phi_n: 
\Omega\times \{n\}\to \C$, $n\in \N$, such that $\phi_n\big(x^{(n)}_m\big)=\phi_0\big(y^{(n)}_m\big)$ for 
each $m=1,\dots,2^n$, where we commit a mild abuse of notation\,---\,repeated below\,---\,of denoting $\phi_n(z,n)$ simply as $\phi_n(z)$, $z\in \Omega$.
\smallskip

Secondly, we present a pair of inequalities that will be needed in the subsequent steps. Combining 
Lemmas~\ref{L: upperbound} and~\ref{L: lowerbound} with the contractivity property of 
the function $p$ as defined in Lemma~\ref{L: covering_map_lemma}, there exists $N\in \Z_+$ such 
that for each $m\geq N$, we have
\[
  1-\frac{K(R)}{m} \leq C_\Omega^*(\sqrt R, R^{1-\tfrac{1}{m}})\leq |x(m)|\leq 1-\frac{2}{m+1}.
\]
 It follows that for $m=2^n, \dots, 2^{n+1}-1$, $n$ sufficiently large,
\[
  1-\frac{K(R)}{2^n} \leq C_\Omega^*\big(\sqrt R,R^{1-\frac{1}{m}}\big)\leq
|x(m)|\leq 1-\frac{2}{2^{n+1}}.
\]
Hence, for each $n$ sufficiently large, we have the following inequality
\begin{equation}\label{E: final_inequality}  
  \bigg(1-\frac{K(R)}{2^n}\bigg)^{2^n} \leq \prod_{m=2^n}^{2^{n+1}-1}C_\Omega^*\big(\sqrt 
  R,R^{1-\frac{1}{m}}\big) \leq \prod_{m=2^n}^{2^{n+1}-1}|x(m)| \leq \bigg( 1
  -\frac{2}{2^{n+1}}\bigg)^{2^n}.
\end{equation}

As $\Omega$ is $C_\Omega$-hyperbolic, there exists a Carath{\'e}odory extremal for the points $x^{(n)}_m$, 
$\sqrt R\in \Omega$.
Fix such an extremal and call it $f_{m,n}$; then $f_{m,n}\big(x^{(n)}_m\big)=0$ and 
$|f_{m,n}(\sqrt R)|=C_\Omega^*\big(\sqrt R, 
x^{(n)}_m \big)$ for each $n\geq 1$, $m=1, \dots,2^n$. Define $f_n:\Omega \to \D$ by 
\begin{equation}\label{E: prod_cara_extrem}
  f_n(z):=\prod_{m=1}^{2^n}f_{m,n}(z).
\end{equation}
\medskip

\noindent{\textbf{Step 2.}\emph{ Let $\big(f_{n_j}\big)_{j\geq 1}\subset \big(f_n\big)_{n\geq 1} $ be a 
locally uniformly convergent subsequence of $\big(f_n\big)_{n\geq 1}$ and assume that the limit is $f$. Then 
$f$ has no zeros.}}
\smallskip

\noindent{As $\big(f_{n_j}\big)_{j\geq 1}$ converges locally uniformly, $\big(f_{n_j}
(\sqrt R)\big)_{j\geq 1}$ is also a convergent sequence in $\C$ and its limit is greater than or equal to 
$1/e^{K(R)}$ by~\eqref{E: final_inequality}. So $f$ is not identically zero on $\Omega$. Let $z_0 \in
\Omega$ be such that $f(z_0)=0$. Let $r>0$ such that $f(z)\neq 0$ on $|z-z_0|=r$ and 
$\overline{B(z_0,r)}\subset \Omega$.
\smallskip

If $y\in f_n^{-1}\{0\}\setminus\mathscr{S}_{n}$ for some $n\in \Z_+$, then there exists $m$ with
$m=1,\dots, 2^n$ such that $y \in f^{-1}_{m,n}\{0\}\setminus \{x^{(n)}_m\}$. By the definition of the 
M{\"o}bius pseudodistance and Lemma~\ref{L: lowerbound}, there exists $N$ such that for each $n \geq N$, we have
\begin{equation}\label{E: eqn_of_0}
  C^*_\Omega(\sqrt R, y) \geq C^*_\Omega\big(\sqrt R, x^{(n)}_m\big) \geq 1- \frac{K(R)}
  {2^n}.
\end{equation}
So all zeros of $f_n$ lie outside the ball $B_{C^*_\Omega}\big(\sqrt R, 1-\frac{K(R)}{2^n}\big)$ for 
each $n\geq N$. As the right hand side of~\eqref{E: eqn_of_0} converges to $1$, there exists 
$M\in \Z_+$ such that 
\[
  B(z_0, r) \cap \bigg(\Omega\setminus B_{C^*_\Omega}\bigg(\sqrt R, 1-\frac{K(R)}
  {2^n}\bigg)\bigg) = \emptyset
\]
for each $n\geq M$. So $B(z_0, r)$ does not contain any zero of 
the function $f_n$ for all $n\geq\max\{M,N\}$. This contradicts Hurwitz's Theorem because $f(z_0)=0$, which establishes Step~2.
}

\medskip
\noindent{\textbf{Step 3.}\emph{ $X(R)$ is Carath{\'e}odory hyperbolic.}}
\smallskip

\noindent{Let $\pi: X(R)\to \Omega$ be the holomorphic function given by $X(R)\ni [z, n] \longmapsto z$ (where $[z, n]$ is as introduced in Step~$1$).
\smallskip

For the points $x\neq y \in X(R)$ with $\pi(x)\neq\pi(y)$, $C_{X(R)}(x,y)$ is positive due to the contractivity property of $\pi$ and Carath{\'e}odory hyperbolicity of $\Omega$.
\smallskip

For the points $x\neq y$ such that $\pi(x)=\pi(y)$, our analysis consists of two cases.

\medskip
\noindent{\textbf{Case 1.} \emph{$x$ and $y$ are such that $Q^{-1}\{x\}$ and $Q^{-1}\{y\}$ are both singleton sets.}}
\smallskip

\noindent{There exist $m_1\neq m_2\in \N$ with $m_1>0$ such that $Q^{-1}\{x\}\in \Omega \times \{m_1\}$ and $Q^{-1}\{y\}\in
\Omega \times \{m_2\}$. Fix a holomorphic function $h:\Omega \times \{m_1\} \to \D $ such that $h\big((x^{(m_1)}_k,m_1)\big)=0$ for each 
$k=1,\dots, 2^{m_1}$ and $h\big(Q^{-1}\{x\}\big)\neq 0$. Now, define a sequence of functions $\phi_n: \Omega 
\times \{n\} \to \D$ for each $n\in \N$ by
\[
  \phi_n(z):= \begin{cases}
              h(z), &\text{if $z\in \Omega \times \{m_1\},$}\\
              0, &\text{otherwise}.
\end{cases}
\]
As discussed above, this sequence induces a holomorphic function $F:X(R) \to \D$ with $F(x)\neq0$ 
and $F(y)=0$.
}

\medskip
\noindent{\textbf{Case 2.} \emph{$x$ and $y$ are such that either $Q^{-1}\{x\}$ or $Q^{-1}\{y\}$ consists of two elements.}}
\smallskip

\noindent{Since $\pi(x)=\pi(y)$, we may assume that $Q^{-1}\{x\}$ has two elements and $Q^{-1}\{y\}$ is singleton. So 
\[
  \pi(x)=\pi(y)=x^{(n)}_m  \text{ and } Q^{-1}\{x\}=\{\big(x^{(n)}_m, n\big),\big(x^{(n)}_m, 0\big)\}
\]
for some $n\in \Z_+$ and $m=1,\dots,2^n$. Also, there exists $m_1\in\Z_+ $ with $n\neq m_1$ such that
\[
  Q^{-1}\{y\}=\big(x^{(n)}_m, m_1\big)\neq \big(x^{(m_1)}_k, m_1\big) \text{ for each 
  $k=1,\dots,2^{m_1}$}.
\]
Now, by an argument similar to that in Case~$1$, we obtain a $\D$-valued holomorphic function
on $X(R)$ that separates $x$ and $y$.
\smallskip

For $x\neq y$ as above, in either case, clearly $C_{X(R)}(x,y)>0$.
}
}

\medskip
\noindent{\textbf{Step 4.}\emph{ $X(R)$ is $C_{X(R)}$-complete but not $C_{X(R)}$-finitely compact}}
\smallskip

\noindent{To prove that $X(R)$ is $C_{X(R)}$-complete, it is useful to know that $\Omega$ is $C
_\Omega$-finitely compact. This fact follows from Result~\ref{R: Ahlfors} together with \cite[Theorem~7.4.1]{jarnickipflug:idca93}.
\smallskip

Let $(x_n)_{n\geq 1}$ be a $C_{X(R)}$-Cauchy sequence and define $A:=\bigcup_{n\geq 1} Q^{-1}\{x_n\}$. Owing to the
contractivity property of $\pi$ defined in Step~$3$, we have for any $\varepsilon >0$, $N\in \Z_+$
such that
\[
  C_\Omega (\pi(x_n), \pi(x_m))\leq C_{X(R)}(x_n, x_m)\leq \varepsilon \quad  \forall m,n \geq N.
\]
As $\Omega$ is $C_{\Omega}$-finitely compact, there exists a compact set $K$ such that $\pi(x_n) \in K$ for each $n\in \Z_+$.

\medskip
\noindent{\textbf{Claim.}}\emph{ The set $\{j\geq 1: A\cap \big(\Omega\times \{j\}\big) \neq \emptyset\}$ is finite.} 
\smallskip

\noindent{If the claim is not true, we can assume that for each 
$n\ge1 $ there exists $m_n\in \Z_+$ with $m_n<m_{n+1}$, and $Q^{-1}\{x_n\}\in \Omega\times\{m_n\}$. Now, if we define 
$a_n:=x_{2n}$ and $b_n:=x_{2n+1}$}, then $C_{X(R)}(a_n, b_n)\to 0$. Define a sequence of 
functions $F_n : X(R) \to \D$ for each $n\in \N$ induced by $\big(F_{n, k}\big)_{k\geq 1}$, where $F_{n, k} : \Omega\times\{k\}\to 
\D$ for each $k\in \Z_+$ is defined by
\begin{equation*}
  F_{n, k}(z, k):=\begin{cases}
                   f_n(z), & \text{if } k=n,\\
                   0, & \text{otherwise}
\end{cases}
\end{equation*}
where $f_n$ is defined in~\eqref{E: prod_cara_extrem}. 
As $C_{X(R)}(a_n, b_n)\to 0$, we have $|F_{m_{2n}}(a_n)|=|f_{m_{2n}}\circ\pi(x_{2n})|\to 0.$ As $\big(\pi(x_n)\big)_{n\geq 1}\subset 
K$, we can assume that $\pi(x_{2n}) \to a_0\in \Omega$ (passing to a subsequence and relabelling, if necessary). By
Montel's theorem, one may extract a further subsequence of $\left(f_{m_{2n}}\right)_{n\geq 1}$ converging uniformly
on compact subsets of $\Omega$ to a holomorphic function $f$. It then follows that $f(a_0)=0$, which contradicts
the assertion in Step~2. Hence the claim.
\hfill$\blacktriangleleft$
\smallskip

So there exist $j_1,\dots,j_k\in\N$ such that $A \subset (\pi\circ Q)^{-1}(K)\cap\big(\bigcup_{i=1}^k \Omega\times
\{j_i\}\big)$. Hence there exists $i_0$ with $1\leq i_0 \leq k$ and $\big(x_{n_s}\big)_{s\geq 
1}\subset(x_n)_{n\geq 1}$ such that $ Q^{-1}\{x_{n_s}\}\cap\big(\Omega\times\{j_{i_0}\}\big) \neq \emptyset$ 
for each $s\in \Z_+$. By the definitions of $Q$ and $\pi$, $(\pi\circ Q)^{-1}(K)\cap \Omega \times \{j_i\}$ 
is compact for all $i=1,\dots, k$. Thus, we may assume that (passing to a subsequence and relabelling, if 
necessary) $x_{n_s} \overset{X(R)}{\longrightarrow} x_0$. As $C_{X(R)}(x_0, \bcdot)$ is continuous, 
$x_{n_s}\overset{C_{X(R)}}{\longrightarrow}x_0$. As $(x_n)_{n\geq 1}$ is $C_{X(R)}$-Cauchy, we have 
$x_n\overset{C_{X(R)}}{\longrightarrow} x_0$. 
\smallskip

Now, if we take any other subsequence $\big(x_{n_l}\big)_{l\geq 1} \subset (x_n)_{n\geq 1}$, then by an 
argument similar to the one in the previous paragraph, we get $\big(x_{n_{l_t}}\big)_{t\geq 1} 
\subset\big(x_{n_l}\big)_{l\geq 1}$ and $y_0\in X(R)$ such that
\[
  Q^{-1}\{x_{n_{l_t}}\}\cap \big(\Omega \times \{j_{i_{1}}\}\big)\neq \emptyset \text{ for some $1\leq i_1 
  \leq k$, for each $t\in \Z_+$,} \quad\text{and}\quad  x_{n_{l_t}}\overset{X(R)}{\longrightarrow} y_0.
\]

By the continuity of $C_{X(R)}(y_0, \bcdot)$ and 
Carath{\'e}odory hyperbolicity of $X(R)$, we have $x_0=y_0$. So we have shown that any subsequence of 
$(x_n)_{n\geq 1}$ has a further subsequence which is convergent and converges to the same limit. Hence
$x_n\overset{X(R)}{\longrightarrow}x_0$. Thus $X(R)$ is $C_{X(R)}$-complete.
\smallskip

Finally, we will show that $X(R)$ is not $C_{X(R)}$-finitely compact. Here, we adapt an argument 
in \cite{jarnickipflugvigue:aeccbnfca93}, but as $\Omega \not\cong \D$, we need the inequality~\eqref{E: 
final_inequality} in place of the classical Schwarz Lemma in \cite{jarnickipflugvigue:aeccbnfca93}. To this 
end, let $g: X(R) \to \D$ be a 
holomorphic function such that $g\big(Q(\sqrt R,0)\big)=0$. Define $\Phi_n: \Omega \to \D$ by 
\begin{equation}\label{E: auxilary_holo_fns}
  \Phi_n(z):= \frac{g\big(Q(z,n)\big)-g\big(Q( z, 0)\big)}{2} \quad \forall n\in\N.
\end{equation}
We have $\Phi_n: \Omega \to \D$ holomorphic function and $\Phi_n$ vanish on 
$\mathscr{S}_{n}$ for all $ n\geq 1$, as defined in~\eqref{E: attaching_points}. Borrowing notation from 
Part~$2$ of Lemma~\ref{L: covering_map_lemma}, denote
$x(2^n+m-1)$ by $\widetilde{x}^{(n)}_m$ for each $n\in\Z_+$ and $m=1,\dots,2^n$. So we have
$\Phi_n\circ p\big(\widetilde{x}^{(n)}_m\big)=0$ for each $n\in\Z+$ and $m=1,\dots,2^n$. Now define $g_n:
\D \to\D$ by
\begin{equation*}
  g_n(z):= \prod_{m=1}^{2^n} \frac{z-\widetilde{x}^{(n)}_m}{1- \overline{\widetilde{x}^{(n)}_m}z} \quad \forall n\in 
  \Z_+.
\end{equation*}
 We have $(\Phi_n \circ p)/g_n: \D \to \D$ is a holomorphic function for each $n\in\Z_+$. From the last 
 inequality in~\eqref{E: final_inequality} and from~\eqref{E: auxilary_holo_fns}, for each  sufficiently 
 large $n\in \Z_+$, since $p(0) = \sqrt{R}$, we have:
\begin{equation*}
  \big|\Phi_n (\sqrt R) \big| \leq \prod_{m=1}^{2^n}\big|\widetilde{x}^{(n)}_m\big| \leq \frac{1}{e},
  \text{ which implies } |g\big(Q(\sqrt R,n)\big)|= 2|\Phi_n(\sqrt R)|\leq\frac{2}{e}.
\end{equation*}
Pick some $r\in(2/e, 1)$. 
As $g$ is an arbitrary holomorphic function from $X(R)$ to $\D$ with $g\big(Q(\sqrt R, 0)\big)=0$, we have
\[
  C^*_{X(R)} \big(Q(\sqrt R,0), Q(\sqrt R,n)\big) \leq \frac{2}{e}  <r \quad \text{for all sufficiently large } 
  n\in \Z_+.
\]
Thus $\dball{C}{*}{X(R)}(Q(\sqrt R, 0), r)$ is not relatively compact inside $X(R)$.
\smallskip

Finally, we note that the example in \cite{jarnickipflugvigue:aeccbnfca93} is contractable whereas,
clearly, $X(R)$ is not. Thus, the two examples are not even homeomorphic, hence they are not
biholomorphically equivalent. This completes the proofs of the properties claimed for $X(R)$.
}
\medskip

\section{Proofs of our theorems}\label{S: pf_thm}
Now we are in a position to present the proofs of our main results. We begin with the proof of Theorem~\ref{T: strngly_complte}. We will use the notation and conventions introduced in Section~\ref{SS: nota_conv}. 
\emph{Importantly,
in what follows, the conventions stated in $(5)$ of Section~\ref{SS: nota_conv} will be in effect.}

\begin{proof}[The proof of Theorem~\ref{T: strngly_complte}]
We will show that if $X$ is $C_X$-complete, then $X$ satisfies condition~$(1)$ in 
Lemma~\ref{L:equvlnce_condn}. To this end, let $K \varsubsetneq X$ be a compact set such that
$K^\mathrm{o} \neq \emptyset$ and fix $z \in K^{\mathrm{o}}$. Assume there does not exist any
$r > 0$ such that $\dball{C}{}{X}(z, r_z) \subseteq K$. So there exists $z_n \in \dball{C}{}{X}
(z, 1/n) \cap (X \setminus K)$ for each $n\in \Z_+$; this implies that, as $z_n 
\overset{C_X}{\longrightarrow} z$, $(z_n)_{n \geq 1}$ is a $C_X$-Cauchy sequence. As $X$
is $C_X$-complete, there exists a point $z_0 \in X$ such that $z_n \overset{X}
{\longrightarrow} z_0$. As $z_n \in X \setminus K$ for all $n\in \Z_+$, we have $z_0 \in X\setminus 
K^{\mathrm{o}}$. By the continuity of $C_X(z_0, \bcdot)$, $z_n \overset{C_X}{\longrightarrow} 
z_0$. As $X$ is $C_X$-hyperbolic, we have $z=z_0$. This contradicts the fact that $z\in 
K^{\mathrm{o}}$. Condition~$(1)$ in Lemma~\ref{L:equvlnce_condn}, and thus the result, follows.
\end{proof}

Next, we prove Theorem~\ref{T: uncountable_example}. We use the fact that two annuli with inner radius $1$ and outer radii $R_1$, $R_2$ are biholomorphic if and only if $R_1 = R_2$.

\begin{proof}[The proof of Theorem~\ref{T: uncountable_example}]
Let $X(R_1)$ and $X(R_2)$, $R_1 \neq R_2$, be two analytic spaces whose construction is given in Section~\ref{S:
example}, with $X(R_i)$ constructed out of copies of the annulus $A(R_i)$, $i=1,2$. Assume there exists a 
biholomorphism $F: X(R_1) \to X(R_2)$. We will use the notation introduced in Section~\ref{S: example}, with
$[z, k]_i$ denoting the equivalance class of $(z, k)\in A(R_i)\times \{k\}$ in $X(R_i)$, $i=1,2$. We have 
natural injective holomorphic maps $i_k : A(R_1)\times \{k\}\to X(R_1)$ and $j_k : A(R_2)\times \{k\}\to X(R_2)$,
\[
  i_k(z, k):=[z, k]_1 \text{ and } j_k(z, k):=[z,k]_2,
\]
for each $k\in \N$. As $F$ is an open map and a local homeomorphism, we have the following:
\begin{itemize}
  \item For $\alpha \in X(R_1)$ with $Q_1^{-1}\{\alpha\}$ a singleton, $Q_2^{-1}\{F(\alpha)\}$ is also a singleton,
  \item For $\alpha \in X(R_1)$ with $Q_1^{-1}\{\alpha\}$ having two elements, $Q_2^{-1}\{F(\alpha)\}$ also has two elements,
  \end{itemize}
where $Q_1$ and $Q_2$ are the relevant quotient maps.  
\smallskip

Let us define for $n\in \Z_+$, $B_{i,n}:=\{[x^{(n)}_{i,m}, n]_i:m=1,\dots,2^n\}\subset X(R_i)$, where 
$x^{(n)}_{i,m}:=R_i^{1-1/(2^n+m-1)}$, $i=1,2$. Let us denote $C_{1,k}:=\big(i_k\big(A(R_1)\times
\{k\}\big)\setminus B_{1,k}\big) \subset X(R_1)$, $k\in \Z_+$, and $C_{1,0}:=\left(i_0\big(A(R_1)\times
\{0\}\big)\setminus \bigcup_{k=1}^\infty B_{1,k}\right) \subset X(R_1)$. In a similar way we can define 
$C_{2,k}\subset X(R_2)$ for each $k\in \N$. It follows from the bullet-points above that $F(C_{1,k})\subset 
X(R_2)\setminus \bigcup_{n=1}^\infty B_{2,n}$ for each $k\in \N$. As $F(C_{1,k})$ is connected for each $k\in 
\N$, there exists $l(k)\in\N$ such that
\begin{equation}\label{E: step_before_biholo}
  F(C_{1, k}) \subseteq C_{2, l(k)}.  
\end{equation}
Also $F(C_{1,k})$ is open for each $k\in\N$. If $k_1 \neq k_2$, then $C_{1, k_1}\cap C_{1, k_2}=\emptyset$, by construction of $X(R_1)$; thus, by injectivity 
of $F$, $F(C_{1,k_1})\cap F(C_{1, k_2}) = \emptyset$. As $F$ is a surjective and open mapping, and $C_{2,n}$ is connected for each 
$n\in \N$, we have equality in~\eqref{E: step_before_biholo}. Hence $C_{1, k}$ and $C_{2,l(k)}$ are homemorphic 
via $F$; this implies $k=l(k)$. To see this, first note that when $j\in \Z_+$, both
$C_{1,\,j}$ and $C_{2,\,j}$ are annuli with $2^j$ points removed, and their fundamental groups are free
groups of rank $2^j+1$. When $j=0$, since $C_{1, 0}$ and $C_{2, 0}$ have infinitely many punctures, and their
fundamental groups are not finitely generated. Thus, as $C_{1, k}$ is homeomorphic to $C_{2,l(k)}$,
$l(k)\neq 0$ if $k\in \Z_+$ and by the above discussion, $k= l(k)$. By the same discussion, if $k=0$,
then $l(k)=0$.
\smallskip

Let $\pi_i: X(R_i)\to A(R_i)$ for $i=1,2$ be as introduced at the begining of Step~$3$ of the construction in 
Section~\ref{S: example}. As $\pi_2 \circ F \circ \left(\left. i_1\right|_{\,A(R_1)\times \{1\}\setminus
\{(x^{(1)}_{1,1},1), (x^{(1)}_{1,2},1)\}}\right)$ is a holomorphic function into $A(R_2)$, 
by Riemann's Removable Singularity Theorem, it extends to a holomorphic map on $A(R_1)\times \{1\}$. But since, by 
the conclusion of the previous paragraph, $\pi_2 \circ F \circ i_1: A(R_1)\times \{1\} \to A(R_2)$ is a 
continuous bijection, $\pi_2 \circ F \circ i_1$ induces a biholomorphism between $A(R_1)$ and $A(R_2)$, which 
contradicts the fact that $R_1 \neq R_2$.
\smallskip

So, there are uncountably many biholomorphically inequivalent analytic spaces. Step~$4$ in Section~\ref{S: 
example} shows that they are not $C_X$-finitely compact but are $C_X$-complete. Finally, from Theorem~\ref{T: strngly_complte} the conclusion of this theorem follows. 
\end{proof}

\section*{Acknowledgements}
I would like to thank my thesis advisor, Prof. Gautam Bharali, 
for many invaluable discussions during the course of this work. I am also grateful to him for his advice on
the writing of this paper. I am also grateful to the anonymous referee for their comments towards improving 
this paper. This work is supported in part by a scholarship from the Prime Minister's Research 
Fellowship (PMRF) programme (fellowship no.~0202593) and by a DST-FIST grant (grant no.~DST FIST-2021 [TPN-700661]).


\smallskip



\begin{thebibliography}{88}
\bibitem{krantz:tgtcv25}
Peter V. Dovbush and Steven G. Krantz, The Geometric Theory of Complex Variables, Springer, Cham, 2025.

\bibitem{gunningRossi:afscv65}
Robert C. Gunning and Hugo Rossi, Analytic Functions of Several Complex Variables, Prentice-Hall, Inc., 
Englewood Cliffs, NJ, 1965.

\bibitem{jarnickipflug:idca93}
Marek Jarnicki, Peter Pflug, Invariant Distances and Metrics in Complex Analysis, De Gruyter Exp. Math., 9, 1993.

\bibitem{jarnickipflugvigue:aeccbnfca93}
Marek Jarnicki, Peter Pflug, Jean-Pierre Vigu{\'e}, {\em An example of a Carath{\'e}odory complete but not finitely compact analytic space}, Proc. Amer. Math. Soc. {\bf 118} (1993), no.~2, 537-539.

\bibitem{sibony: pdfhbmc75}
Nessim Sibony, {\em  Prolongement des fonctions holomorphes born{\'e}es et m{\'e}trique de Carath{\'e}odory}, Invent. Math. {\bf 29} (1975), no.~3, 205-230.

\end{thebibliography}
\end{document}